\documentclass{fundam}

\usepackage[ruled,lined]{algorithm2e}
\usepackage[utf8]{inputenc}
\usepackage{graphicx}
\usepackage{tikz}
\usetikzlibrary{automata, positioning, arrows}
\usepackage{amsmath,amssymb}
\usepackage{url}

\newtheorem{que}{Question}[section]
\newtheorem{thm}[que]{Theorem}
\newtheorem{defi}[que]{Definition}
\newtheorem{cl}[que]{Claim}
\newtheorem{obs}[que]{Observation}

\newenvironment{proof1}[2] {\paragraph{Proof of {#1} {#2} :}}{\hfill$\square$}

\newcommand{\cF}{\mathcal{F}}
\newcommand{\cP}{\mathcal{P}}
\newcommand{\cT}{\mathcal{T}}
\newcommand{\cH}{\mathcal{H}}

 \usepackage{hyperref}

\begin{document}


\setcounter{page}{83}
\publyear{22}
\papernumber{2104}
\volume{185}
\issue{1}

    \finalVersionForARXIV


\title{Spanning Structures in Walker--Breaker Games}

\author{Jovana Forcan\thanks{Also affiliated at: Department of Informatics, Faculty of Philosophy,
                           University of East Sarajevo, Bosnia and Herzegovina. \newline
                    Address for correspondence: Department of Mathematics and Informatics,
                   Faculty of Sciences, University of Novi Sad.},  Mirjana Mikala\v{c}ki\thanks{The author's research is partially supported by
                      Ministry of Education, Science and Technological Development, Republic of Serbia,
                 Grant No. 451-03-68/2022-14/200125.\newline \newline
          \vspace*{-6mm}{\scriptsize{Received February 2020; \ accepted January 2022.}}}
     \\
Department of Mathematics and Informatics \\
Faculty of Sciences, University of Novi Sad \\
Trg Dositeja Obradovi\'{c}a 4, 21000 Novi Sad, Serbia \\
dmi.jovana.jankovic@student.pmf.uns.ac.rs,  mirjana.mikalacki@dmi.uns.ac.rs
}

\maketitle

\runninghead{J.\ Forcan and M.\ Mikala\v{c}ki}{Spanning Structures in Walker--Breaker Games}

\begin{abstract}
We study the biased $(2:b)$ Walker--Breaker games, played on the edge set of the complete graph on $n$ vertices, $K_n$. These games are a variant of the Maker--Breaker games with the restriction that Walker (playing the role of Maker) has to choose her edges according to a walk. We look at the two standard graph games -- the Connectivity game and the Hamilton Cycle game and show that Walker can win both games even when playing against Breaker whose bias is of the order of magnitude $n/ \ln n$.
\end{abstract}


\section{Introduction}
We study Walker--Breaker games, a variant of the well-known \textit{Maker--Breaker} positional games, recently introduced by Espig, Frieze, Krivelevich and Pegden in~\cite{EFKP15}. Given two positive integers, $a$ and $b$, a finite set $X$ and $\cF\subseteq 2^X$, in the biased $(a:b)$ Maker--Breaker game $(X,\cF)$, two players, \emph{Maker} and \emph{Breaker} take turns in claiming previously unclaimed elements of $X$ until all of them are claimed, with Maker going first. The parameters $a$ and $b$ are referred to as biases of Maker, respectively Breaker, the set $X$ is referred to as \emph{the board} of the game, and the members of $\cF$ are called the \emph{winning sets}. When there is no risk of confusion and $X$ is known, we use only $\cF$ to denote the game $(X,\cF)$. Maker wins the game $\cF$ if, by the end of the game, she claims all elements of some $F \in \cF$. Breaker wins otherwise.

Maker--Breaker games are one of the most studied representatives of positional games and more about these games and some others can be found in~\cite{BeckBook,HKSSbook}.
In the most common setup, the Maker--Breaker games are played on the edge set of the complete graph on $n$ vertices, $K_n$, where the winning sets are some graph-theoretic structures, like spanning trees, Hamilton cycles, triangles, etc. In particular, we look at two standard games: the \emph{Connectivity game}, $\cT_n$, and the \emph{Hamilton Cycle game}, $\cH_n$. In the Connectivity game $\cT_n$, Maker's goal is to claim a spanning tree, and in the Hamilton Cycle game $\cH_n$, Maker's aim is to claim a Hamilton cycle. When $a=b=1$ (\emph{unbiased game}) it is known that Maker can win easily in both of the games (see~\cite{Lehman, HKSS2, HS09}) and in fact in most of the standard graph games on $K_n$, for sufficiently large $n$. That motivated the study of biased $(1:b)$ games, first introduced in the seminal paper of Chv\'{a}tal and Erd\H{o}s~\cite{CE78}. In the same paper, Chv\'{a}tal and Erd\H{o}s observed that Maker--Breaker games are \emph{bias monotone}. That means that if the $(1:b)$ game is a Breaker's win for some value of $b$, then the $(1:b+1)$ game is also a Breaker's win. This property enabled the definition of the \emph{threshold bias} of the game $(X,\cF)$, which is the unique integer $b_{\cF}$ such that for all values of $b<b_{\cF}$, the $(1:b)$ game $(X,\cF)$ is a Maker's win and for all values $b\geq b_{\cF}$, the $(1:b)$ game $(X,\cF)$ is a Breaker's win.
For both of the games we are interested in, the Connectivity game $\cT_n$, and the Hamilton Cycle game $\cH_n$, the threshold bias is of the same order of magnitude, i.e.\ $b_{\cT_n}=\Theta\left(\frac{n}{\ln n}\right)$ and $b_{\cH_n}=\Theta\left(\frac{n}{\ln n}\right)$.

In \cite{CE78}, Chv\'{a}tal and Erd\H{o}s showed that the threshold bias for the Connectivity game is between $(1/4 - \varepsilon)n/ \ln n$ and $(1 + \varepsilon)n/ \ln n$ for any $\varepsilon>0$.
The upper bound of the threshold bias for the Hamilton Cycle game is the same as in the Connectivity game.
Since a disconnected graph cannot contain a Hamilton cycle, in the $(1:b)$ Hamilton Cycle game Breaker can apply the same strategy as in the $(1:b)$ Connectivity game.
A slightly better lower bound for the Connectivity game is given latter by Beck \cite{Beck82} who improved a constant factor to $\ln 2 $. Finally, Gebauer and Szab\'{o} in~\cite{GS09} showed that Maker can win in the Connectivity game if $b\leq (1 - \varepsilon)n/\ln{n}$, for every $\varepsilon > 0$, and in this way proved that the threshold bias for the Connectivity game on $K_n$ is asymptotically equal to $n/ \ln n$.
Relying on the approach and results from \cite{GS09}, Krivelevich in~\cite{K11} showed that the threshold bias for the Hamilton Cycle game is asymptotic to $n / \ln n$, by proving that Maker can win  if $b\leq \left(1-\frac{30}{\ln^{1/4}{n}}\right)\frac{n}{\ln{n}}$.

Doubly biased $(a:b)$ Maker—Breaker Connectivity game and Hamilton Cycle game are studied in \cite{HMS11, MM13} where the lower and upper threshold biases in both games are established for all relevant values of $a$, and for most values these bounds are quite sharp.
Although $(a:b)$ Maker--Breaker games are studied much less then unbiased games or $(1:b)$ games, there are examples of games where just a slight change in bias changes the outcome of the game. One such example is the \emph{diametar-}$2$ game, i.e.\ the game where the board is $E(K_n)$ and the winning sets are all spanning subgraphs of $K_n$ with the diameter at most $2$. This game, when played in $(1:1)$ setup, is known to be Breaker's win. However, it is shown in~\cite{BMP09} that $(2:2)$ diameter-$2$ game is Maker's win. There are also other examples of the so-called doubly biased games (see e.g.\ \cite{BeckBook, HKSSbook, BMP09}).

Now, as mentioned at the beginning, we are interested in the Walker--Breaker games, where Walker (playing the role of Maker) is restricted by the way of choosing her edges. Namely, she has to choose her edges according to a walk, i.e.\ for her starting position $v$ she can choose any vertex, and when it is her turn to play, she can claim any edge $vw$ incident to $v$ not previously claimed by Breaker (but it can be previously claimed by herself). After that, the vertex $w$ becomes her current position.
The study of this type of games was initiated by Espig, Frieze, Krivelevich and Pegden in~\cite{EFKP15} and further developments were made by Clemens and Tran in~\cite{CT16}.
\noindent When playing with bias $1$, Walker cannot claim any spanning structure for any $b\geq 1$ in the $(1:b)$ Walker--Breaker games, as Breaker can easily isolate a vertex in her graph by choosing in each move an edge between some fixed vertex, say $u$, isolated in Walker's graph, and Walker's current position and another $b-1$ arbitrary edges also incident with $u$. Therefore, to help Walker, we increase the bias by just one, and look at two $(2:b)$ Walker--Breaker games. It turns out that this small increase in the bias makes a big difference for Walker in terms of claiming a spanning structure in a graph.

\medskip
 More specifically, we focus on two questions raised in~\cite{CT16}:
\begin{que}(\cite{CT16}, Problem 6.4)\label{q1}
What is the largest bias $b$ for which Walker has a strategy to create a spanning tree of $K_n$ in the $(2 : b)$ Walker--Breaker game on $K_n$?
\end{que}

\begin{que}(\cite{CT16}, Problem 6.5)\label{q2}
Is there a constant $c > 0$ such that Walker has a strategy to occupy a Hamilton cycle of $K_n$ in the $(2 : \frac{cn}{\ln{n}})$ Walker--Breaker game on $K_n$?
\end{que}
\noindent In this paper, we answer both questions, and obtain that in the biased $(2:b)$ version of both Walker--Breaker games -- the Connectivity and the Hamilton Cycle -- the threshold bias is of the order $\frac{n}{\ln n}$, which corresponds to the $(1:b)$ Maker--Breaker version of the games.
\noindent To be able to answer Question~\ref{q1}, we need the following two theorems. The first one gives the lower bound for the threshold bias in the $(2:b)$ Walker--Breaker Connectivity game.

\begin{thm} \label{tConn}
For every $0<\varepsilon < \frac{1}{4}$ and every
large enough $n$, Walker has a strategy to win in the biased $(2 : b)$ Walker--Breaker Connectivity game played on $K_n$, provided that $b \leq  \left(\frac{1}{4} - \varepsilon \right) \frac{n}{\ln{n}}.$
\end{thm}
Theorem~\ref{tBreaker} provides the upper bound for the threshold bias in the $(2:b)$ Connectivity game. Its proof is is very similar (actually almost identical) to the proof of Theorem~$3.1$ by Chv\'{a}tal and Erd\H{o}s~\cite{CE78}, as their proof can be easily adjusted to comply with Walker--Breaker rules and Walker's bias $2$. For the completeness of the paper, the proof of Theorem~\ref{tBreaker} is given in the Appendix.

\begin{thm} \label{tBreaker}
For every $\varepsilon >0$ and $b \geq (1+\varepsilon )\frac{n}{\ln{n}}$, Breaker has a strategy to win in the $(2:b)$ Walker--Breaker Connectivity game on $K_n$, for large enough $n$.
\end{thm}
The following theorem answers Question~\ref{q2} and gives the lower bound for the threshold bias in the $(2:b)$ Walker--Breaker Hamilton Cycle game.

\begin{thm} \label{tHam}
There exists a constant $\alpha > 0$ for which for every large enough $n$ and $b \leq \alpha \frac{n}{\ln{n}}$, Walker has a winning strategy in the $(2 : b)$ Hamilton Cycle game played on $K_n$.
\end{thm}
Since being the first player in Maker--Breaker is always an advantage, it is common to assume that the player for which a winning strategy is provided is the second player to move. When providing a winning strategy for Walker in Theorem~\ref{tConn} and Theorem~\ref{tHam},
we suppose that Breaker is the first player, i.e.\  one \textit{round} in the game consists of a move by Breaker followed by a move of Walker. To prove Theorem~\ref{tBreaker} we provide Breaker with the strategy that works even if Walker is the first player. \\

The rest of the paper is organized as follows. In Section~\ref{prel} we list the tools necessary for proving Theorem~\ref{tConn} and Theorem~\ref{tHam}. In Section~\ref{cg} we prove Theorem~\ref{tConn} and in Section~\ref{Hamiltonicity} we prove Theorem~\ref{tHam}. Finally, in Section~\ref{cr} we conclude with some remarks.

\subsection{Notation}
Our notation is standard and follows that of \cite{West01}. Specifically, we use the following. \smallskip

For a given graph $G$ by $V(G)$ and $E(G)$ we denote its vertex set and edge set, respectively.
The order of a graph $G$ is denoted by $v(G)=|V(G)|$, and the size of the graph by $e(G)=|E(G)|$.  \smallskip

For $X,Y \subseteq V(G)$, let $E(X, Y) := \{xy \in E(G) : x \in X, y \in Y\}$ and let $N(X,Y) := \{y\in Y : \exists x\in X \text{ such that } xy \in E(G) \}$.  By $N(X) := \{v \in V(G) : \exists x \in X \text{ such that } vx \in E(G)\} $ we denote the neighbourhood of $X$.
When $X$ consists of a single vertex $x$, we abbreviate $N(x)$ for $N(\{x\})$ and $N(x, Y)$ for $N(\{x\}, Y)$. Let $d_G(x) = |N(x)|$ denote the degree of vertex $x$ in $G$.
Given two vertices $x,y\in V(G)$ an edge in $G$ is denoted by $xy$.  \smallskip

Assume that the Walker--Breaker game, played on the edge set of a graph $G$, is in progress. At any given moment during this game, we denote the graph spanned by Walker's edges by $W$ and the graph spanned by Breaker's edges by $B$. For some vertex $v$ we say that it is \emph{visited} by a player if he/she has claimed \emph{at least one} edge incident with $v$. A vertex is \emph{isolated/unvisited} if no edge incident to it is claimed. The edges in $E(G \setminus(W\cup B))$ are called \emph{free}.  \smallskip

Let $n$ be a positive integer and let $0 \leq p:= p(n) \leq 1$. The Erd\H{o}s-R\'{e}nyi model $\mathbb{G}(n,p)$ is a random subgraph $G$ of $K_n$, constructed by retaining each edge of $K_n$ in $G$ independently at random with probability $p$.
We say that graph $G\sim \mathbb{G}(n, p)$ possesses a graph property $\mathcal{P} $  \textit{asymptotically almost surely}, or a.a.s., for brevity, if the probability that $\mathbb{G}(n, p)$ possesses $\mathcal{P} $ tends to 1 as $n$ goes to infinity. Throughout the paper, we will use the approximation $\sum_{i=1}^n\frac{1}{i}\leq \ln n +1$.

\section{Preliminaries}
\label{prel}

For the analysis of Walker's winning strategy in the Connectivity game, we need the \emph{Box game}, first introduced by Chv\'{a}tal and Erd\H{o}s in \cite{CE78}. The rules are as follows.
The\emph{ Box game} $Box(k, t, a,1)$ is played on $k$ disjoint winning sets, whose sizes differ by at most 1, that contain altogether $t$ elements. The players in the Box game will be called \emph{BoxMaker} and \emph{BoxBreaker}.
BoxMaker claims $a$ elements per move, while BoxBreaker claims 1 element per move. BoxMaker wins if and only if she succeeds to claim all elements of some winning set. Otherwise, BoxBreaker wins. Chv\'{a}tal  and Erd\H{o}s in \cite{CE78} defined the following recursive function:
$$f(k,a): = \left\{
\begin{array}{rl}
0, & k=1\\
\left \lfloor \frac{k(f(k-1,a)+a)}{k-1} \right \rfloor , & k\geq 2.
\end{array}
\right.$$
The value of $f(k,a)$ can be approximated as
$$(a-1)k\sum _{i=1}^{k-1}\frac{1}{i} \leq f(k,a) \leq ak\sum _{i=1}^{k-1}\frac{1}{i}.$$

\noindent The following theorem from~\cite{CE78} gives the criterion for BoxMaker's win in $Box(k,t,a,1)$.

\begin{thm} (\cite{CE78}, Theorem 2.1, \textbf{the Box game criterion})\label{boxgame}
Let $a$, $k$ and $t$ be positive integers. BoxMaker has a winning strategy in $Box(k,t,a,1)$ if and only if $t\leq f(k,a)$.
\end{thm}
\noindent In order to answer Question \ref{q2}, we need some statements related to local resilience and random graphs.

\begin{defi} (\cite{CT16}, Definition 2.1)
For $n \in \mathbb{N}$, let $\mathcal{P} = \cP(n)$ be some graph property that is monotone increasing, and let
$0 \leq \varepsilon, p = p(n) \leq 1$. Then $\cP$ is said to be $(p,\varepsilon )$-resilient if a random graph $G \sim \mathbb{G}(n,p)$ a.a.s. has the following property: For every $R \subseteq G$ with $d_R(v) \leq \varepsilon d_G(v)$ for every $v \in V (G)$ it holds that $G \setminus R \in \cP$.
\end{defi}

\noindent The next theorem provides a good bound on the local resilience of a random graph with respect to Hamiltonicity.

\begin{thm} (\cite{LS12}, Theorem 1.1) \label{lr}
For every positive $\varepsilon > 0$, there exists a constant $C = C(\varepsilon )$ such that for $p \geq \frac{C\ln{n}}{n}$, a graph $G \sim \mathbb{G}(n,p$) is a.a.s. such that the following holds. Suppose that $H$ is a subgraph of $G$ for which $G' = G - H$ has minimum degree at least $(1/2 +\varepsilon )np$, then $G'$ is Hamiltonian.
\end{thm}
\noindent The proof of Theorem \ref{tHam} will follow from Theorem \ref{lr} and the following statement, which is the key ingredient. We will prove both of them in Section \ref{Hamiltonicity}.

\begin{thm} \label{pt}
For every constant $0 < \varepsilon  \leq 1/100$ and a sufficiently large integer $n$ the following holds. Suppose that $\frac{10\ln{n}}{\varepsilon n} \leq p < 1$ and $\cP$ is a monotone increasing graph property which is $(p,\varepsilon)$-resilient. Then in the $\left( 2: \frac{\varepsilon}{60p}\right)$ game on $K_n$ Walker has a strategy to create a graph that satisfies~$\cP$.
\end{thm}
 To prove Theorem \ref{pt} we will use an auxiliary \emph{MinBox} game which is motivated by the study of the degree game~\cite{GS09}. The $MinBox(n, D, \alpha , b)$ game is a Maker--Breaker game played on $n$ disjoint boxes $F_1,...,F_n$ each of order at least $D$. For our needs, we will suppose that Breaker starts the game, i.e.\ in one round Breaker claims $b$ elements and then Maker claims 1 element.
Maker wins the game if she succeeds to claim at least $\alpha |F_i|$ elements in each box $F_i$, $1\leq i \leq n$.

\eject
The number of elements in box $F$ claimed so far by Maker and Breaker are denoted by $w_M(F)$ and $w_B(F)$, respectively.
The box $F$ is \emph{free} if there are elements in it still not claimed by any of the players.
If $w_M(F ) < \alpha |F|$, then $F$ is an \emph{active} box.
For each box $F$ we set the \emph{danger} value to be $\mathrm{dang}(F) := w_B(F ) - b \cdot w_M(F)$.

\medskip
 We also need the following upper bound on the danger value.

\begin{thm} (\cite{FKN15}, Theorem 2.3) \label{theorem}
Let $n, b, D \in \mathbb{N}$, let $0 < \alpha < 1$ be a real number, and consider the game $MinBox(n, D, \alpha , b)$. Assume that Maker plays as follows: In each turn, she chooses an arbitrary free active box with maximum danger, and then she claims one free element from this box. Then, proceeding according to this strategy,
$$\mathrm{dang}(F) \leq b(\ln{n} + 1)$$
is maintained for every active box $F$ throughout the game.
\end{thm}

\section{The connectivity game}
\label{cg}

\begin{proof1}{Theorem} {\ref{tConn}}
First, we present a winning strategy for Walker and then prove that during the game she can follow it. By $U \subseteq  V (K_n)$ we denote the set of vertices, not yet visited by
Walker, which is dynamically maintained throughout the game. At the beginning of the game, we have $U := V (K_n)$.

\paragraph{Walker's strategy.} In the first round Walker visits three vertices. She identifies two vertices $v_0$ and $v_1$ with the largest degrees in Breaker's graph. Let $d_B(v_0) \geq d_B(v_1)$. She starts her move in vertex $v_0$ and then, if $v_0v_1 \in E(B)$, she finds a vertex $u \in U$ such that the edges $v_0u$ and $uv_1$ are free, and claims them. Otherwise, if $v_0v_1 \notin E(B)$, she claims $v_0v_1$ and then from $v_1$
moves to some $u' \in  U$ such that $d_B(u') = \mathrm{max}\{d_B(u) : u \in U\}$ (ties broken arbitrarily) and $v_1u'$ is free.  \smallskip

In every round $r\geq 2$ Walker visits at least one vertex from $U$.
After Breaker's move, Walker identifies a vertex $a \in U$ such that $d_B(a) =\mathrm{max}\{d_B(u) : u \in U\}$ (ties broken arbitrarily).
Then Walker checks if there is some vertex $y\in U$ such that edges
$wy$ and $ya$ are free, where $w$ is Walker's current position, and she claims these two edges $wy$ and $ya$. If no such vertex $y\in U$ exists, then Walker finds an arbitrary vertex $y' \in V(K_n)$, which could be already visited by Walker, such that edges $wy'$ and $y'a$ are free. She claims these two edges.   \smallskip

Assuming Walker can follow this strategy, she plays at most
$n-2$ rounds.  \smallskip

In the remainder of this proof, we will show that Walker can follow the proposed strategy.

\medskip
First, we are going to consider the maximum degree of unvisited vertices in Breaker's graph $B$.
We can analyse Walker's strategy through an auxiliary Box game, where she takes the role of BoxBreaker, and plays as a second player.
In the Box game each box represents all free edges adjacent to some vertex $u\in U$. At the beginning of the game, the number of boxes is $n$ and the
number of elements in each box is $n-1$. Boxes are not disjoint since for any $u\in U$, vertices from $N(u)$ can also belong to $U$.
So, each edge of the original game belongs to two of these boxes. BoxBreaker can pretend that the boxes are disjoint and that BoxMaker claims $2b$ elements in every move. So, we look at the Box game $Box(n, n(n-1), 2b, 1)$.
We estimate the size of the largest box that BoxMaker could fill within at most $n-2$ rounds.
This gives us the maximum degree of vertices from $U$ in Breaker's graph $B$.

\medskip
The size of the largest box is at most
\begin{equation} \label{eq1}
\frac{2b}{n} + \frac{2b}{n-1} + ... + \frac{2b}{3} \leq 2b(\ln{n}+1)- \left(\frac{2b}{2} + \frac{2b}{1}\right) = 2b \ln{n} - b.
\end{equation}
Now, we are going to prove that Walker can follow her strategy. The proof goes by induction on the number of rounds. After Breaker's first move we have that $d_B(v_0)+d_B(v_1) \leq b+1$, so it is obvious that Walker can visit vertices $v_0$ and $v_1$. Suppose that Walker already played $k \leq n-3$ rounds and visited at least $k+2$  vertices.
Suppose that Walker finished this round at some vertex $w$ and at the end of this round $d_B(w) \leq 2b \ln{n} -b$.

\medskip
According to (\ref{eq1}), after Breaker's move in round $k+1$, some vertex $a \in U$ can have degree at most $2b \ln{n} - b$ in $B$. Also, Breaker could have claimed all $b$ edges incident with $w$, in his $(k+1)^{\mathrm{st}}$ move, so $d_B(w)\leq 2b\ln n$.
Walker finds some vertex $y'\in V(K_n)$ such that edges $wy'$ and $y'a$ are free, with preference that $y' \in U$.
Such vertex exists since
$$d_B(w) + d_B(a) \leq 4b \ln{n} - b < n-2$$
So, Walker is able to play her move in $(k+1)^{st}$ round.
\end{proof1}

\section{The Hamilton cycle game}
\label{Hamiltonicity}

In this section we prove Theorem~\ref{pt} and Theorem~\ref{tHam}.
The proof of Theorem \ref{pt} follows very closely to the proofs of Theorem 1.5 in \cite{FKN15} and Theorem 2.4 in \cite{CT16}.

\begin{proof1}{Theorem} {\ref{pt}}
We show that Walker has a strategy to build a graph that satisfies property $\mathcal{P}$.
Walker's strategy will be partly deterministic and partly random.  \smallskip

In the random part of the strategy, Walker generates a random graph $H \sim \mathbb{G}(n, p)$ on the vertex set $V(K_n)$, by tossing a biased coin on each edge of $K_n$ (even if this edge already belongs to $E(B)$), independently at random, which succeeds with probability $p$.
When Walker tosses a coin for an edge $e$, we say that she \textit{exposes} the edge $e$.
For each vertex $v\in V(K_n)$ we consider the set $U_v \subseteq N(v,V(K_n))$ which contains those vertices $u$ for which the edge $vu$ is still not exposed. At the beginning of the game, $U_v = N(v,V(K_n))$ for all $v \in V(K_n)$.

\medskip
To decide for which edges she needs to toss a coin, Walker identifies an \textit{exposure vertex} $v$ (the way of choosing the exposure vertex will be explained later).
If her current position is different from $v$, she needs to play her move deterministically. That is, she finds two edges $wy$ and $yv$, where $w$ is her current position and $y$ is some vertex from $V(K_n)$, such that $wy, yv \notin E(B)$.
She claims these edges and after that move, $wy, yv \in E(W)$, where by $W$ we denote a graph spanned by all Walker's edges.

Once she comes to the exposure vertex $v$, she starts tossing her coin for edges incident with $v$ with the second endpoint in $U_v$ in the arbitrary order, until she has a first success or until all edges incident with $v$ are exposed.
Every edge $e$ on which Walker has success when tossing her coin is included in $H$. If this edge $e$ is free, Walker claims it.
If the exposure failed to reveal a new edge in $H$, she declares her move a failure of type I. If she has success on an edge, but that edge belongs to $E(B)$, she declares her move a failure of type II.   \smallskip

Let $G'$ be a graph containing all the edges in $H\cap W$.

\medskip
We need to prove that $G'\in \cP$. In order to do this we need to show that following her strategy, Walker will achieve
that a.a.s. $d_{G'}(v)\geq (1-\varepsilon )d_H(v)$ holds for each $v\in V(K_n)$, where $0<\varepsilon \leq 1/100$.
Since $H$ is random, the degree of $v$ in $H$ can be determined by using Chernoff's inequality~\cite{ASBook08}. We have
$$\mathbb{P} \left [ Bin(n-1, p) < \frac{99}{100}(n-1)p \right ] = o\left( \frac{1}{n} \right)$$
for $p \geq \frac{10\ln{n}}{\varepsilon n}$. Thus, by the union bound, it holds that a.a.s.
$$d_H(v) \geq \frac{99}{100}(n-1)p$$
for all vertices in $V(K_n)$.

\medskip
So, if we prove that the number of failures of type II is relatively small, at most $\varepsilon d_H(v) $, we get that
$d_{G'}(v) \geq (1-\varepsilon ) d_H(v)$ for every $v\in V(K_n)$.

\medskip
Let $f_I(v)$ and $f_{II}(v)$ denote the number of failures of type I and type II, respectively, for the exposure vertex $v$.

To keep the number of failures of type II small enough, Walker simulates an auxiliary $MinBox(n$, $4n,p/2,4b)$ game in which she takes the role of Maker.
To each vertex $v \in V(K_n)$ Walker assigns the box $F_v$ of size $4n$ at the beginning of the game. For each box $F$ danger value is defined by
$\mathrm{dang}(F)=w_B(F)-4b\cdot w_M(F)$.

\medskip
We describe Walker's strategy in detail.\vspace*{-2mm}

\paragraph{Walker's strategy. } Walker's strategy is divided into two stages.\vspace*{-2mm}

\paragraph{Stage 1.}
After every Breaker's move, she updates the simulated $MinBox(n,4n,p/2,4b)$ game, in the following way: for each of $b$ edges $pq$ that Breaker claimed in his move, Walker assumes that he claimed one free element from $F_p$ and one from $F_q$.

\medskip
In the first round, Walker first identifies a free active box with the largest danger. Let $F_v$ be such a box. Maker claims an element of $F_v$ in the $MinBox(n,4n,p/2,4b)$ game and Walker selects $v$ as the exposure vertex.
Walker starts her move in some vertex $v_0 \in V(K_n)$, $v_0 \neq v$, and then finds some vertex $v_1 \in V(K_n)$, such that edges $v_0v_1$ and $v_1v$ are free. This is possible since $d_B(v_0) + d_B(v) \leq b+1$.
In the second round, Walker starts the exposure process on the edges $vv'$, $v'\in U_v$, that is, proceeds with Case 2 (see case distinction below).

In every other round $r\geq 3$, Walker plays in the following way.
Denote by $w$ Walker's current position and suppose that it is her turn to make a move.
First, she updates the simulated $MinBox(n,4n$, $p/2,4b)$ game, as described above. After that, she checks whether an exposure vertex exists and proceeds with the case distinction below.
Otherwise, she finds a vertex $v$ such that in the simulated $MinBox(n,4n,p/2,4b)$ game $F_v$ is a free active box of the largest danger. If no free box exists, then Walker proceeds to Stage 2. Otherwise, she declares the vertex $v$ as the new exposure vertex, Maker claims an element of $F_v$ in the $MinBox(n,4n,p/2,4b)$ game and then in the real game Walker proceeds with the following cases.\vspace*{-2mm}

\paragraph{Case 1.} $w\neq v$. Walker finds a vertex $y \in V(K_n)$ such that edges $wy$ and $yv$ are free or belong to $E(W)$, where $v$ is the new exposure vertex. Then, she moves to vertex $v$ using these edges. If these edges were free and Walker claimed them, then these edges are now part of the Walker's graph $W$. Walker proceeds with Case 2.\vspace*{-2mm}

\paragraph{Case 2.} Vertex in which Walker is currently positioned is the exposure vertex. Let $\sigma : [|U_v|] \rightarrow U_v$
be an arbitrary permutation on $U_v$. She starts tossing a biased coin for vertices in $U_v$, independently at random with probability of success $p$, according to the ordering of $\sigma $.\vspace*{-2mm}

\paragraph{2a.} If this coin tossing brings no success, she increases the value of $f_I(v)$ by 1 and in the simulated game $MinBox(n, 4n, p/2 , 4b)$ Maker claims $2pn - 1$ additional free elements from $F_v$ or all remaining free elements if their number is less than $2pn - 1$. She updates $U_v = \emptyset $ and removes $v$ from all other $U_{\sigma (i)}$ for each $i \leq |U_v|$. In the real game Walker moves along some edge which is in $E(W)$
and then returns to $v$ by using the same edge.\vspace*{-2mm}

\paragraph{2b.} Suppose that first success happens at the $k^{\mathrm{th}}$ coin toss. Walker declares that $v\sigma (k)$ is an edge of~$H$.
\begin{enumerate}
\itemsep=0.95pt
\item[-] If the edge $v\sigma (k)$ is free, Walker claims this edge and from now on $v\sigma (k)\in E(W)$.
She moves along this edge one more time in order to return to vertex $v$. Also, Walker includes this edge in $G'$. Walker removes $v$ from $U_{\sigma (i)}$ and $\sigma (i)$ from $U_v$, for all $i \leq k$. Maker claims one free element from box $F_{\sigma (k)}$.
\item[-] If the edge $v\sigma (k)$ already belongs to $E(W)$, Walker moves along this edge twice. Also, this edge becomes part of graph $G'$. Walker removes $v$ from $U_{\sigma (i)}$ and $\sigma (i)$ from $U_v$, for all $i \leq k$. Maker also claims one free element from box $F_{\sigma (k)}$.
\item[-] If the edge $v\sigma (k)$ belongs to Breaker, then the exposure is a failure of type II. She increments $f_{II}(v)$ and $f_{II}(\sigma (k))$ by 1. She also updates $U_v := U_v \setminus \{\sigma (i) : i \leq k \} $ and
$U_{\sigma (i)} := U_{\sigma (i)} \setminus \{v\}$ for each $i \leq k$. To make her move, Walker uses an arbitrary edge $vu$ from her graph and returns to $v$ by using the same edge.
\end{enumerate}
At the end of Walker's move in Case 2, the vertex $v$ is not exposure any more.

\paragraph{Stage 2.} Walker tosses her coin on every unexposed edge $uv \in E(K_n)$. In case of success, she declares a failure of type II for both vertices $u$ and $v$.
\eject

\begin{obs}
At any point of Stage 1, there can be at most one exposure vertex.
\end{obs}
\begin{cl} \label{c1}
During Stage 1, Breaker claims at most $4b$ elements in the simulated $MinBox$ game between two consecutive moves of Maker.
\end{cl}

\begin{proof} Suppose that Breaker finished his move in round $t$ and now it is Walker's turn to make her move in this round.
Suppose that in the previous round, $t-1$, Walker played according to her strategy from Case 2.
Let $w$ be Walker's current position.
Walker identifies a free active box $F_v$ which has the largest danger. Maker claims an element from $F_v$. If $w=v$, Walker will start her exposure process on the edges $vv'$ with $v' \in U_v$ in round $t$ and then in the following round, $t+1$, she will again identify a new exposure vertex and Maker will claim an element from the corresponding box. In this case between two Maker's moves, Breaker claims $b$ edges, that is, $2b$ elements from all boxes.
If $w \neq v$, Walker needs to play her move deterministically in order to move from $w$ to $v$ and then in round $t+1$ she will start her exposure process. After she identifies the new exposure vertex in round $t+2$, Maker will claim an element from the corresponding box. In this case between two Maker's moves (in rounds $t$ and $t+2$), Breaker claims $2b$ edges, that is, $4b$ elements from all boxes.
\end{proof}

\begin{cl}  \label{c2}
At any point during Stage 1, we have $w_M(F_v) <(1+2p)n$ and $w_B(F_v) < n$ for every box $F_v$ in the simulated game. In particular, $w_M(F_v) + w_B(F_v) < 4n$, thus no box is ever exhausted of free elements.
\end{cl}

\begin{proof}
According to Walker's strategy, the number $w_M(F_v)$ increases by one every time vertex $v$ is the exposure vertex or when coin tossing brings success on edge $vv'$, where $v'$ is exposure vertex.
There can be at most $n-1$ exposure processes in which Walker can toss a coin on an edge that is incident with $v$. So, both cases together can happen at most $n-1$ times.

\medskip
Also, when Walker declares the failure of type I, $w_M(F_v)$ increases by at most $2pn - 1$.
So, we have
$$w_M(F_v)< n + f_I(v) \cdot 2pn .$$
We claim that failure of type I can happen at most once.
This is true, because after the first failure of type I on $v$,
when Maker receives at most $2pn - 1$ additional free elements from $F_v$, the box $F_v$ is not active any more. So, Maker will never play on vertex $v$ again.
Therefore, $w_M(F_v) < n + 2pn = (1 + 2p)n$.

\medskip
During Stage 1, we have $w_B(F_v) < n$, because
Breaker claims an element of $F_v$ in the simulated
game $MinBox(n,4n,p/2,4b)$ if and only if in the real game he claims an edge incident with $v$. Therefore, $w_M(F_v)+w_B(F_v)< 4n$, as stated.
\end{proof}

\begin{cl} \label{c3}
For every vertex $v \in V (K_n)$, $F_v$ becomes inactive before $d_B(v) \geq  \frac{\varepsilon (n-1)}{5}$.
\end{cl}
\eject
\begin{proof}
Assume that $F_v$ is an active box such that $w_B(F_v) = d_B(v) \geq \frac{\varepsilon (n-1)}{5}$. Since $w_B(F_v) - 4b \cdot w_M(F_v) \leq 4b(\ln{n}+1)$, according to Theorem \ref{theorem}, it follows that $w_M(F_v) \geq \frac{w_B(F_v)}{4b}-(\ln{n}+1)$. With $b = \frac{\varepsilon}{60p}$ 
\noindent we have $w_M(F_v) \geq 3p(n-1) - (\ln{n}+1) > 2pn$, where $p\geq \frac{10\ln{n}}{\varepsilon n}$.
This is a contradiction because $F_v$ is active.
\end{proof}

\begin{cl} \label{c4}
Walker is able to move from her current position to the new exposure vertex.
\end{cl}

\begin{proof}
Let $w$ be Walker's current position at the beginning of some round $t$ and let $v$ be the new exposure vertex. This means that at the beginning of round $t - 1$, the box $F_w$ was active and we had $d_B(w) < \frac{\varepsilon (n-1)}{5}$ . If $F_w$ is no longer active at the end of round $t-1$, then after Breaker's move in round $t$ we have $d_B(w) < \frac{\varepsilon (n-1)}{5}+ b$. We need to show that Walker can find a vertex $y \in V (K_n)$ such that edges $wy$ and $yv$ are not in $E(B)$.

\medskip
Since $F_v$ is free active box and taking into consideration the value of $b$, we have
$$d_B(w) + d_B(v) < \frac{2\varepsilon (n-1)}{5}+ b<n-2$$
and so Walker is able to move to $v$.
\end{proof}

\begin{cl} \label{c5}
For every vertex $v \in V (K_n)$ we have that a.a.s. $F_v$ is active, for as long as $U_v \neq \emptyset $.
In particular, at the end of Stage 1 all edges of $K_n$ are exposed a.a.s.
\end{cl}

\begin{proof}
Suppose that there is a vertex $v$ such that $F_v$ is not an active box and $U_v \neq \emptyset $. Since $F_v$ is not an active box it holds that $w_M(F_v) \geq \frac{p}{2}|F_v|$.
Also, since $U_v \neq \emptyset $, we have that $f_I(v) = 0$.
Maker could increase $w_M(v)$ at the moment when $v$ became the exposure vertex, or when Walker had success on edge $vv'$ where $v'$ is the exposure vertex.
Consider the case when $v$ was the exposure vertex.
Since $f_I(v) = 0$, Walker has success on some edges incident with $v$, besides maybe in the last exposure process at $v$. The total number of successes includes the success for every edge incident with $v$ when $v$ was the exposure vertex, since for that edge the coin will not be tossed again. So, Walker had success on at least $\frac{p}{2}|F_v| - 1 = 2np -1$ edges incident with $v$.

\medskip
Also, every time coin tossing brought success for an edge incident with $v$, the degree of vertex $v$ increased in $H$ by one. It follows that $ d_H(v) \geq \frac{p}{2}|F_v| -1 >2(n-1)p $. By using Chernoff's inequality~\cite{ASBook08}, we have
$$\mathbb{P}[Bin(n-1, p) \geq 2(n-1)p] < e^{-(n-1)p/3} = o\left(\frac{1}{n}\right).$$
Applying the union bound, it follows that with probability $1 - o(1)$, there exists no such vertex. \\
Suppose that at the beginning of Stage 2 there is an edge $uv \in E(K_n)$ which is not exposed. This means that $U_v \neq \emptyset $. So, $F_v$ is an active box and we have that $w_M(F_v) < 2pn$. Since $F_v$ is active it also holds that $w_B(F_v) <\frac{\varepsilon (n-1)}{5}$, according to Claim \ref{c3}.
Therefore, since $w_M(F_v) + w_B(F_v) < |F_v|$, the box $F_v$ is free. But this is not possible at the beginning of Stage 2. A contradiction.
\end{proof}

\begin{cl} \label{c6}
For every vertex $v \in V(K_n)$, a.a.s. we have $f_{II}(v) \leq \frac{9}{10} \varepsilon (n-1)p$.
\end{cl}

\begin{proof}
Failures of type II happen in Stage 1 in case when Walker has success on edge which is in $E(B)$. During Stage 1, by Claim \ref{c3}, for as long as the box $F_v$ is active, for
some $v \in K_n$, we have $d_B(v) < \frac{\varepsilon (n-1)}{5}$ . So, for every $v \in V(K_n)$ there is a non-negative integer
$m \leq \frac{\varepsilon (n-1)}{5}$ such that $f_{II}(v)$ is dominated by $Bin(m,p)$.
Applying a Chernoff's argument~\cite{ASBook08} with $p \geq \frac{10\ln{n}}{\varepsilon n}$ we obtain
$$\mathbb{P}\left[Bin(m, p) \geq \frac{9}{10} \varepsilon (n-1)p \right] \leq \left(\frac{e\varepsilon (n-1)p/5}{\frac{9}{10}\varepsilon (n-1)p}\right)^{\frac{9}{10}\varepsilon (n-1)p} = o\left(\frac{1}{n}\right).$$
The probability that there exists such a vertex is $o(1)$. Thus, a.a.s. $f_{II}(v) \leq \frac{9}{10} \varepsilon (n-1)p$ for all
$v \in V(K_n)$.
\end{proof}

\medskip
 According to Claim \ref{c5}, Walker never played Stage 2 since
she exposed all the edges of $K_n$ by the end of Stage 1.
By Claim \ref{c6} we know that for each vertex $v$ coin tossing has brought success for at most
$$\frac{9}{10} \varepsilon (n-1)p \leq \frac{90}{99} \varepsilon d_H(v) $$
edges which were claimed by Breaker.
So, it follows that for each vertex $v\in V(K_n)$ we have
$$d_{G'}(v)\geq d_H(v) - \frac{90}{99} \varepsilon d_H(v) \geq (1-\varepsilon )d_H(v).$$
This completes the proof of Theorem~\ref{pt}.
\end{proof1}

\begin{proof1}{Theorem} {\ref{tHam}}
When we know that Theorem~\ref{pt} holds, the proof of this theorem is almost the same as the proof of Theorem 1.6 in \cite{FKN15}.

\medskip
Let $C = C\left(\frac{1}{6}\right)$ be as in Theorem~\ref{lr}. Let $p=\frac{c\ln{n}}{n}$ where $c=\mathrm{max}\{C,1000\}$, and let $G \sim \mathbb{G}(n, p)$. Note that the property $\cP$ := ``being Hamiltonian" is $(p,\frac{1}{6})$-resilient for $p\geq \frac{c\ln{n}}{n}$.

\medskip
Applying Chernoff's inequality~\cite{ASBook08}, we obtain
$$\mathbb{P}\left[d_G(v)<\frac{5}{6}np\right] < e^{-\frac{np}{72}} = o\left(\frac{1}{n}\right).$$
Thus, by the union bound, it holds that a.a.s.\  $\delta(G) \geq \frac{5}{6}np$.

\medskip
Let $R \subseteq G$ be a subgraph such that $d_R(v) \leq \frac{1}{6}d_G(v)$.
For $R' = G-R$ we have
$$d_{R'}(v) \geq  \frac{5}{6}d_G(v) \geq \frac{25}{36}np > \frac{2}{3}np = (1/2 + 1/6)np.$$
Theorem~\ref{lr} implies that graph $R'$ is Hamiltonian.

\medskip
According to Theorem~\ref{pt}, Walker can create a graph $G' \in \cP$ in the $\left(2:\frac{1}{360p}\right)$ game on $K_n$.
For $p=\frac{c\ln{n}}{n}$ it follows that Walker has a winning strategy in $\left(2:\frac{n}{360c\ln{n}}\right)$ Walker--Breaker Hamilton Cycle game. Setting $\alpha = \frac{1}{360c}$ completes the proof.
\end{proof1}

\section{Concluding remarks}
\label{cr}

In this paper, we have shown that if we increase Walker's bias by just one, she can win both the Connectivity and the Hamilton Cycle game and moreover, the threshold bias in both games is of the same order of magnitude as in the corresponding Maker--Breaker games. By further analysing the constants obtained, the results for the $(2:b)$ Walker--Breaker Connectivity game and Hamilton Cycle game are closer to the results for their $(1:b)$ Maker--Breaker counterpart.

It would be interesting to determine the order of the threshold bias for the $(a:b)$ Walker--Breaker Connectivity and Hamilton Cycle game when $a>2$.
For $a = 2k$, $k \in \mathbb{N}$ we believe that the same techniques used for $(2:b)$ Walker--Breaker games could be applied, while for $a = 2k+1$, the task would be more challenging. The odd value of Walker's bias can be inconvenient for Walker in the cases when in some rounds she needs to reach some vertex and to stay positioned at that vertex in order to be prepared for her next move in the following round.

\medskip
\textbf{Analyzing other games.} Now that we know that Walker (as Maker) can make spanning structures of $K_n$ even when playing against Breaker whose bias is of the order of magnitude $n/\ln n$, we are curious to find out what happens in other games involving spanning structures.

It is not hard to show that for $b=o(\sqrt{n})$ Walker can win
in the $(2:b)$ Pancyclicity game, that is, she can build a graph which consists of cycles of any given length $3\leq l \leq n$. Indeed, since for $p=\omega(n^{-1/2})$ the property $\cP$ :=``being pancyclic" is $(p,1/2 + o(1))$-resilient (see Theorem 1.1 in \cite{KLS10}), by applying Theorem~\ref{pt} with $p$ and $\cP$ we obtain that Walker has the winning strategy in the $\left(2:\frac{1}{180p}\right)$ Pancyclicity game on $K_n$.

We wonder what happens in the $k$-Connectivity game (where the winning sets are all $k$-connected graphs), for $k\geq 2$ and what would be the largest value of $b$ for which Walker can win the $(2:b)$ Walker--Breaker $k$-Connectivity game on $K_n$.

\medskip
\textbf{Different board.} Another question that comes naturally is what happens if we change the board to be the edge set of a general graph $G$ or some sparse graph. How many vertices could Walker visit then in both unbiased and biased games?

\appendix

\section{Proof of Theorem~\ref{tBreaker}}
\noindent To prove Theorem~\ref{tBreaker} we need to provide Breaker with a strategy which will enable him to isolate a vertex from Walker's graph for given bias $b\geq (1+\varepsilon )\frac{n}{\ln{n}}$. For that, we rely on the strategy of Breaker in the $(1:b)$ Maker--Breaker Connectivity game~\cite{CE78}, where Breaker first makes a clique in his graph and then isolates one of the vertices from that clique in Maker's graph. Looking from Breaker's point of view, in the Connectivity game Walker claiming two edges per move can achieve the same as Maker claiming one edge per move. Therefore, in order to win in the $(2:b)$ Walker--Breaker Connectivity game Breaker can apply the same strategy as Breaker in the $(1:b)$ Maker--Breaker Connectivity game~\cite{CE78}.

\begin{proof}
Suppose that Walker begins the game. Breaker's winning strategy is divided into two stages.

\paragraph{Stage 1.} Breaker builds a clique $C$ of order $m=\left \lfloor \frac{b}{2} \right \rfloor $ such that all vertices from $C$ are isolated in Walker's graph.

\paragraph{Stage 2.} Breaker isolates one of the vertices from $C$ in Walker's graph.

\medskip
Now we are going to prove that Breaker can follow his strategy.

\paragraph{Stage 1. }
Breaker will play at most $b/2$ moves.
Suppose that in round $i-1$, where $i\leq b/2$, Breaker built a clique $C_{i-1}$, such that all its vertices are isolated in Walker's graph.
After Walker's move in round $i$, Walker's graph contains at most $2i$ edges and at most $2i+1$ vertices.
Since $i<n/2-2$, there are at least two vertices $u$ and $v$ outside the Breaker's clique which are not incident with Walker's edges.

\smallskip
Then Breaker can claim the edge $uv$ and $2(i-1)$ edges joining $uv$ to $V(C_{i-1})$. In this way he creates clique $C_{i}$ of order $|C_{i-1}| +2$. In the round $i+1$, Walker can visit only one vertex from $C_i$.
After Walker visits some $c\in C_i$, Breaker's graph still contains a clique $C'$ isolated in Walker graph with $V(C')=C_i \setminus \{c\}$.

\paragraph{Stage 2.} Let $C$ be the Breaker's clique of order $m = |C|$ after Stage 1. Let $c_1,c_2,...,c_m \in C$. To isolate some vertex $c_i \in C$ Breaker needs to claim $n-m$ edges $c_iu$, $u\in V(K_n) \setminus V(C)$. In each round Walker can visit at most one vertex from $C$, so she will need to play at most $m$ rounds in Stage 2 to visit all vertices from $C$.
We can use an auxiliary Box Game $Box(m, m \cdot (n-m), b,1)$ to estimate whether Breaker can isolate a vertex from his clique in Walker's graph in at most $m$ moves.
Breaker is the BoxMaker who claims $b$ elements per move. Walker, assuming to play the role of BoxBreaker, can claim an element in at most one unvisited box per move. Note that here, BoxBreaker is the first to play.

\smallskip
It can be verified that $m(n-m)<(b-1)m\ln{(m-1)}$ holds for given $b$ and $m$, and therefore the condition in Theorem~\ref{boxgame} is satisfied, so BoxMaker can win the game.
This means that Breaker is able to isolate a vertex in Walker's graph and thus he wins in the $(2:b)$ Walker--Breaker Connectivity game.
\end{proof}

\end{document}